\documentclass[11pt]{article}

\usepackage[margin=1.08in]{geometry}
\usepackage[T1]{fontenc}
\usepackage[utf8]{inputenc}
\usepackage{lmodern}
\usepackage{microtype}
\usepackage{amsmath,amssymb,amsthm,mathtools}
\usepackage{enumitem}
\usepackage{ytableau}
\usepackage{tikz}
\usetikzlibrary{arrows.meta,positioning}
\usepackage[dvipsnames]{xcolor}
\usepackage[colorlinks=true,linkcolor=MidnightBlue,citecolor=ForestGreen,urlcolor=BrickRed]{hyperref}
\hypersetup{
  pdftitle={Dual Weight and Monodromy of Dual Affine RS Correspondence},
  pdfsubject={The dual weight, common monodromy, and its relation to the AMBC weight}
}

\newtheorem{theorem}{Theorem}[section]
\newtheorem{proposition}[theorem]{Proposition}
\newtheorem{lemma}[theorem]{Lemma}
\newtheorem{corollary}[theorem]{Corollary}
\theoremstyle{definition}
\newtheorem{definition}[theorem]{Definition}

\newtheorem{example}[theorem]{Example}
\theoremstyle{remark}
\newtheorem{remark}[theorem]{Remark}

\newcommand{\Z}{\mathbb Z}
\newcommand{\Sn}{\widetilde S_n}
\newcommand{\RSd}{\operatorname{DARS}}
\newcommand{\AMBC}{\operatorname{AMBC}}
\newcommand{\ol}[1]{\overline{#1}}
\newcommand{\row}{\mathfrak{row}}
\newcommand{\RSYT}{\textup{RSYT}}
\newcommand{\yttab}[1]{%
  \ytableausetup{boxsize=.6cm,aligntableaux=center}%
  \begin{ytableau}#1\end{ytableau}}
\newcommand{\ind}{\operatorname{ind}}
\newcommand{\lch}{\operatorname{lch}}
\newcommand{\xch}{\operatorname{xch}}
\newcommand{\RRSS}{\operatorname{RRSS}}
\newcommand{\Dom}{\operatorname{Dom}}
\newcommand{\e}{\mathbf e}
\newcommand{\Rop}{\mathsf R}
\newcommand{\Lop}{\mathsf L}
\newcommand{\E}{\mathsf E}
\newcommand{\A}{\mathcal A}

\title{Dual Weight and Monodromy of Dual Affine RS Correspondence}
\author{
Yifeng ZHANG\thanks{
Email: \tt calvinz314159@gmail.com
}
}
\date{South China Normal University}

\begin{document}
\maketitle

\begin{abstract}
The dual affine Robinson--Schensted correspondence and the affine matrix-ball construction give two related parametrizations of extended affine permutations.
From the stable-window data of the dual correspondence, we introduce a dual weight \(\beta\) and prove that it is consistent with the original pair \((\lambda,N_0)\).
We prove that \(\beta\) and the AMBC weight \(\rho\) have identical monodromy along affine Knuth paths.
We further give an explicit relation between \(\beta\) and \(\rho\), showing that their difference depends only on the associated tabloids.
\end{abstract}

\section{Introduction}\label{sec:introduction}

The classical Robinson--Schensted correspondence converts a permutation into a pair of
standard Young tableaux of the same shape.  Besides its original insertion-theoretic
form, it admits diagrammatic realizations through Fomin's growth diagrams and the
Viennot--Fulton matrix-ball viewpoint.  The different models emphasize complementary
features: insertion is effective for computation, while local growth rules and shadow
lines make symmetries of the correspondence especially visible
\cite{Fom95,Vie77,Ful97}.  On
the algebraic side, the recording tableau, insertion tableau, and common shape
parametrize the left, right, and two-sided Kazhdan--Lusztig cells in type~A
\cite{KL79}.  On the geometric side, Robinson--Schensted can be recovered from relative
positions of flags and the Jordan types induced on their intersections
\cite{Ste88,BF01}.  These interpretations make the correspondence a meeting point of
algebraic combinatorics, Hecke-algebra cells, and flag geometry.

The affine theory retains this three-way interaction but introduces genuinely infinite
data.  Shi constructed an affine Robinson--Schensted map from affine permutations to
pairs of tabloids and proved that its insertion tabloid detects right cells
\cite{Shi91}.  Since the set of affine permutations is infinite whereas the set of
tabloid pairs of a fixed size is finite, this map cannot be injective.  Honeywill
completed the tabloid pair by a dominant weight \cite{Hon05}.  Chmutov, Pylyavskyy, and
Yudovina then gave a bijective and diagrammatic refinement, the affine matrix-ball
construction (AMBC) \cite{CPY18}.  AMBC decomposes the periodic permutation matrix into
streams, records their domain and image residues as two tabloids, and stores their
relative altitudes in a dominant integral vector \(\rho\).  It also relates the affine
correspondence to asymptotic ordinary Robinson--Schensted insertion; the Weyl-group
symmetry in its inverse construction explains why the weight is subject to dominance.

Huang and Zhang introduced a dual affine Robinson--Schensted correspondence using
periodic growth diagrams and intrinsically colored shadow lines \cite{HZ26}.  Their
construction simultaneously extends the Fomin--Viennot diagrammatic models and is dual
to Shi's correspondence and AMBC: the insertion tabloids agree, while the recording
tabloids are related by affine evacuation \cite{CFKLY22}.  It therefore recovers the
known affine Kazhdan--Lusztig cell parametrization and suggests a prospective geometric
interpretation in terms of relative positions of affine flags.  Its output has the form
\[
                 \RSd(w)=(\bar P,\bar Q,\lambda,N_0),
\]
where \(\bar P,\bar Q\) are tabloids of common shape \(\mu\), \(\lambda\) is the
partition in the first stable window, and \(N_0\) records the position of that window.
The tabloid comparison leaves a natural question: how should the stable-window pair
\((\lambda,N_0)\) be compared with the single AMBC weight \(\rho\)?

The following two theorems summarize the three main results of this paper.

\begin{theorem}\label{thm:intro-beta}
The dual affine correspondence admits a dual weight \(\beta\) that is independent of
the stable window.  The data \((\beta,N_0)\) and the original
stable-window data \((\lambda,N_0)\) determine each other.
\end{theorem}

\begin{theorem}\label{thm:intro-comparison}
After identifying the corresponding recording tabloids, the dual weight \(\beta\) and
the AMBC weight \(\rho\) have the same monodromy.  Moreover, their difference depends
only on the associated tabloids; equivalently, there is a tabloid correction
\(\kappa(P,Q)\) such that
\[
                         \beta=\rho+\kappa(P,Q).
\]
\end{theorem}

The paper has six sections.  Section~\ref{sec:prelim} fixes the conventions for affine
permutations, tabloids, stable windows, AMBC, and affine evacuation.
Section~\ref{sec:beta} introduces \(\beta\) and proves its consistency with the original
\((\lambda,N_0)\)-data.  Section~\ref{sec:operations} studies natural transformations of
the dual correspondence.  Section~\ref{sec:comparison} establishes the monodromy and
comparison results for \(\beta\) and \(\rho\).  Section~\ref{sec:cells} records the
cell-theoretic consequences.

\section*{Acknowledgments}
We thank Daoji Huang for useful discussion about dual affine RS correspondence.

\section{Preliminaries}\label{sec:prelim}

\subsection{Extended affine permutations and windows}

Fix a positive integer \(n\).  An \emph{extended affine permutation} is a bijection
\(w:\Z\to\Z\) satisfying
\[
                         w(i+n)=w(i)+n
                         \qquad(i\in\Z).
\]
It is determined by its \emph{window} \([w(1),\ldots,w(n)]\).  Conversely, an
integer list \([a_1,\ldots,a_n]\) is the window of an extended affine permutation
if and only if the residues of the \(a_i\) form a complete residue system modulo
\(n\).  We write \(\Sn\) for the extended affine symmetric group, with multiplication
given by composition, and set
\begin{equation}\label{eq:index}
        \ind(w):=\frac1n\sum_{i=1}^n\bigl(w(i)-i\bigr)\in\Z.
\end{equation}
The integrality follows by comparing the two complete residue systems
\(w(1),\ldots,w(n)\) and \(1,\ldots,n\) modulo \(n\).  The index is a group
homomorphism \(\Sn\to\Z\).  Its zero fiber is the affine Weyl group, and we write
\(\Sn^{(k)}\) for the index-\(k\) fiber.  Thus the index records the central
translation component that is absent in the ordinary affine symmetric group.

For \(i\in\Z/n\Z\), let \(s_i\) be the affine simple reflection that exchanges every
pair of positions congruent to \(i\) and \(i+1\) modulo \(n\).  Thus right
multiplication by \(s_i\) exchanges the values in those positions, whereas left
multiplication exchanges the corresponding value residues \(\ol i\) and \(\ol{i+1}\).
This convention also
fixes the cyclic simple reflection \(s_0=s_n\), which crosses the chosen window boundary.

Let \(\tau(i)=i+1\).  The two elementary window transformations used below are
\begin{align}
 \Rop w&:=w\circ\tau
      =[w(2),\ldots,w(n),w(1)+n],
 \label{eq:R-def}\\
 \Lop w&:=\tau\circ w
      =[w(1)+1,\ldots,w(n)+1].
 \label{eq:L-def}
\end{align}
Right translation changes the chosen domain interval, while left translation shifts
all image values.  Directly from the windows,
\begin{equation}\label{eq:index-shifts-prelim}
             \ind(\Rop w)=\ind(\Lop w)=\ind(w)+1.
\end{equation}
The two operations commute and satisfy
\begin{equation}\label{eq:central-window}
                         \Rop^n w=\Lop^n w=w+n.
\end{equation}
Here \(w+n\) denotes the central translate \(i\mapsto w(i)+n\).  Although
\(\Rop^n\) and \(\Lop^n\) agree, their intermediate effects on insertion and
recording data are different; this is the reason for keeping both operations.

We also use the Dynkin reflection \(r\), characterized by
\(r(s_i)=s_{n-i}\), and the reflected inverse
\begin{equation}\label{eq:iota}
                         \iota(w):=r(w^{-1}).
\end{equation}
On the periodic permutation matrix, inversion is reflection in the main diagonal,
whereas \(\iota\) is reflection in the antidiagonal of a fundamental square.

\begin{example}[Index and the two window directions]\label{ex:prelim-window}
For \(n=4\), the list
\[
                         w=[10,3,-3,12]
\]
defines an extended affine permutation because its entries have residues
\(2,3,1,0\) modulo \(4\).  Its index is
\[
               \ind(w)=\frac{10+3-3+12-(1+2+3+4)}4=3.
\]
The two elementary transforms are
\[
 \Rop w=[3,-3,12,14],
 \qquad
 \Lop w=[11,4,-2,13],
\]
and both have index \(4\).  Four iterations of either transformation give
\([14,7,1,16]=w+4\), illustrating \eqref{eq:central-window}.
\end{example}

\subsection{Tabloids, rotation, and affine evacuation}

Let \(\mu=(\mu_1,\ldots,\mu_\ell)\vdash n\), with
\(\mu_1\ge\cdots\ge\mu_\ell>0\).  A \emph{tabloid} of shape \(\mu\) is an
equivalence class of fillings of the Young diagram of shape \(\mu\) with the elements
of \([\ol n]=\{\ol1,\ldots,\ol n\}\), where two fillings are equivalent if one is
obtained from the other by reordering entries within rows.  Here \(\ol a\) denotes the
residue class of \(a\) modulo \(n\).  We regard the \(i\)-th row as a set \(U_i\) and
display every tabloid in row-standard form.  Let \(\RSYT(\mu)\) denote the set of
tabloids of shape \(\mu\).  For an entry \(\ol a\), let \(\row_U(\ol a)\) be the row
containing it, and let \(\e_i\) be the \(i\)-th standard basis vector of \(\Z^\ell\).

Let \(\sigma U\) be obtained by replacing every entry \(\ol a\) with \(\ol{a+1}\).
This rotation preserves the shape and satisfies
\(\sigma^nU=U\).  Affine evacuation is denoted by \(\E\).  It is an involution and
intertwines the two rotation directions by \cite{CFKLY22}
\begin{equation}\label{eq:evac-rotation}
                         \E\sigma^{-1}=\sigma\E.
\end{equation}
Thus a rotation appearing on the AMBC recording side becomes the opposite rotation
after passing to the dual recording tabloid.

We use one distinguished base tabloid repeatedly.  If \(a\in\Z\), the
\emph{reverse row superstandard tabloid} \(\RRSS(\mu,\ol a)\) is obtained by placing
\(\ol a,\ol{a+1},\ldots,\ol{a+\mu_\ell-1}\) in the last row and then continuing with
consecutive residues from bottom to top.  In particular, \(\RRSS(\mu,\ol1)\) begins
with \(\ol1\) in its bottom row.  This is the normalization used for the base point
in the tabloid Knuth graph \cite[Section~2.1]{CLP18}.

We also need to turn a tabloid into a skew tableau.  Let \(\lambda\) be a partition
with at most \(\ell\) parts.  Whenever \(\lambda+\mu\) is a partition, define
\(\lambda\hookrightarrow U\) to be the row-increasing filling of
\((\lambda+\mu)/\lambda\) whose \(i\)-th row has content \(U_i\).  The row order is
forced, but the filling need not be increasing down columns.  Hence the assertion that
\(\lambda\hookrightarrow U\) is a skew standard Young tableau is a genuine
compatibility condition on \(\lambda\) and \(U\).

For later use, define
\begin{align}
 c_U(m)&:=\sum_{a=n-m+1}^{n}\e_{\row_U(\ol a)},
 &d_U(m)&:=\sum_{a=1}^{m}\e_{\row_U(\ol a)}
                         \qquad(0\le m\le n).
\label{eq:countvectors}
\end{align}
The \(i\)-th coordinate of \(c_U(m)\) counts the last \(m\) residues in row \(i\),
whereas that of \(d_U(m)\) counts the first \(m\) residues.  In particular,
\begin{equation}\label{eq:countvectors-total}
                         c_U(n)=d_U(n)=\mu.
\end{equation}
These vectors will be the elementary weight increments for window shifts.

\begin{example}[Rotation and boundary counts]\label{ex:prelim-tabloid}
For \(n=4\) and \(\mu=(2,1,1)\), let
\[
 U=\yttab{\ol1&\ol3\\ \ol2\\ \ol4}.
\]
Then
\[
 \sigma U=\yttab{\ol2&\ol4\\ \ol3\\ \ol1},
 \qquad
 d_U(2)=\e_1+\e_2,
 \qquad
 c_U(2)=\e_1+\e_3.
\]
The equality \(c_U(4)=d_U(4)=(2,1,1)\) records the full row content.
\end{example}

\subsection{AMBC, dominance, and affine Knuth moves}

Place a ball at every lattice point \((i,w(i))\).  The resulting matrix-ball diagram
is invariant under translation by \((n,n)\).  The affine matrix-ball construction
successively decomposes this periodic set into streams.  The stream densities form a
partition \(\mu\), their domain and image residues form two tabloids, and their relative
altitudes form an integral vector.  In this way AMBC gives a bijection
\cite[Theorems~5.1, 5.11, and~5.12]{CPY18}
\begin{equation}\label{eq:AMBC}
                         \AMBC(w)=(P,Q,\rho),
\end{equation}
where \(P,Q\) are tabloids of common shape \(\mu\) and \(\rho\in\Z^\ell\) is dominant
relative to \((P,Q)\).  We use the index normalization of
\cite[Theorem~10.3 and Lemma~10.6]{CPY18}:
\begin{equation}\label{eq:rho-index}
                         \sum_{i=1}^{\ell}\rho_i=\ind(w).
\end{equation}
Thus the weight records both the relative stream positions and the index stratum.
If \(x\in\Z^\ell\) is not dominant for the pair \((A,B)\), let
\(\Dom_{A,B}(x)\) denote the unique dominant representative in the corresponding
AMBC fiber.  Dominant representatives are necessary because permuting streams of equal
density changes their raw altitude list without changing the affine permutation.
The inversion rule is \cite[Proposition~3.1]{CLP18}
\begin{equation}\label{eq:AMBC-inverse}
 \AMBC(w^{-1})=
 \bigl(Q,P,\Dom_{Q,P}(-\rho)\bigr).
\end{equation}

Suppose two adjacent rows \(U_i,U_{i+1}\) have the same length.  Greedily match each
entry of the upper row to the least unused larger entry of the lower row, wrapping around
cyclically when necessary.  The number of wrapped pairs is the local charge
\(\lch_i(U)\).  The result is independent of the order in which the upper-row entries
are processed when they are taken increasingly.  The AMBC dominance inequalities are
\cite[Definition~5.9 and Theorem~5.10]{CLP18}
\begin{equation}\label{eq:AMBC-dominance}
 \rho_{i+1}-\rho_i\ge\lch_i(P)-\lch_i(Q)
 \qquad\text{whenever }\mu_i=\mu_{i+1}.
\end{equation}
There is no adjacent dominance inequality across a strict drop
\(\mu_i>\mu_{i+1}\).  Consequently, the dominance constraints are organized by the
maximal blocks of equal row length.

A right affine Knuth move at residue \(i\in\Z/n\Z\) interchanges the values in the
adjacent affine positions \(i\) and \(i+1\), together with all their periodic
translates, provided that either \(w(i-1)\) or \(w(i+2)\) lies strictly between
\(w(i)\) and \(w(i+1)\).  Equivalently, it is the admissible right multiplication
\(w\mapsto ws_i\).  On the AMBC recording tabloid it exchanges the entries
\(\ol i,\ol{i+1}\), while the insertion tabloid is unchanged.  A move with
\(i\not\equiv0\pmod n\) is \emph{noncyclic}; the move exchanging \(\ol n\)
and \(\ol1\) is \emph{cyclic}.  The cyclic distinction matters because crossing the
chosen window boundary may change the altitude weight.

A \emph{left affine Knuth move} is obtained by applying a right affine Knuth move to
\(w^{-1}\) and then inverting back.  Equivalently, it is an admissible left
multiplication \(w\mapsto s_iw\).  Thus inversion exchanges right and left affine
Knuth moves.

\begin{example}[A dominance inequality]\label{ex:prelim-dominance}
Let \(\mu=(2,2)\) and take
\[
 P=\yttab{\ol1&\ol4\\ \ol2&\ol3},
 \qquad
 Q=\yttab{\ol1&\ol2\\ \ol3&\ol4}.
\]
For \(P\), the entry \(\ol1\) matches \(\ol2\), while \(\ol4\) must wrap before
matching \(\ol3\); hence \(\lch_1(P)=1\).  No wrap is needed for \(Q\), so
\(\lch_1(Q)=0\).  The dominance condition for a compatible AMBC weight is therefore
\(\rho_2-\rho_1\ge1\).  This small example shows why
dominance depends on the tabloid pair rather than only on the shape.
\end{example}

Finally, the tabloid comparison theorem of Huang--Zhang gives
\cite[Theorem~5.10]{HZ26}
\begin{equation}\label{eq:tabloid-comparison}
                         \bar P=P,
                         \qquad \bar Q=\E(Q).
\end{equation}

\subsection{Dual affine RS correspondence}

We recall enough of the periodic growth-diagram construction to identify its four
outputs.  Start with the periodic permutation matrix of \(w\) and apply the dual local
growth rules.  The diagram is tiled by a bi-infinite sequence of \(n\)-by-\(n\)
standard windows \(\Omega_M\), indexed so that translating one period northeast sends
\(\Omega_M\) to \(\Omega_{M+1}\).  Each boundary edge carries both a partition label
and a shadow-line color.

A window is \emph{stable} if its east and west boundary color sequences agree and the
color multiplicities are weakly decreasing with the color.  The same agreement then
holds on the north and south boundaries.  Stability means that periodic shadow lines
have separated into fixed color families; the multiplicities of these families form a
partition \(\mu\vdash n\) \cite[Definition~3.10]{HZ26}.  Once one standard window is
stable, every later standard window is stable \cite[Lemma~4.5]{HZ26}.  We write \(N_0\)
for the least stable index and
\(\lambda[M]\) for the partition at the northeast corner of \(\Omega_M\).  In
particular, \(\lambda=\lambda[N_0]\).

Reading the color positions on the east and north boundaries of the first stable
window produces two tabloids \(\bar P\) and \(\bar Q\), respectively.  The dual affine
correspondence of Huang--Zhang is the bijection \cite[Theorem~3.16]{HZ26}
\begin{equation}\label{eq:DARS-original}
                  \RSd(w)=(\bar P,\bar Q,\lambda,N_0).
\end{equation}
The common shape of \(\bar P\) and \(\bar Q\) is the stable color-multiplicity
partition \(\mu\).  It is useful to spell out the content of ``compatible'' in this
statement.  The data satisfy the following conditions.
\begin{enumerate}[label=(\roman*),leftmargin=2.2em]
\item The partition \(\lambda\) has at most \(\ell(\mu)\) parts, and both
\(\lambda\hookrightarrow\bar P\) and
\(\lambda\hookrightarrow\bar Q\) are skew standard Young tableaux.
\item The pair is reduced: one cannot slide the same column of both skew tableaux up
by one box and retain two skew standard Young tableaux.
\item The colored shadow-line diagram reconstructed from the two skew tableaux is
saturated; in particular, a bump square never supports two shadow lines of different
colors.  Here a \emph{bump square} is a local tile in which two shadow lines meet
without crossing.
\item The predecessor \(\lambda-\mu\), when it is a partition, fails at least one of
the preceding conditions.  This is the condition that \(N_0\) is the \emph{first}
stable index rather than merely a stable index.
\end{enumerate}
The index \(N_0\) itself is an arbitrary integer, but it is tied to \(\lambda\) and the
index of \(w\) by the size formula
\cite[Theorem~3.16]{HZ26}
\begin{equation}\label{eq:dual-size-prelim}
                         |\lambda|=n(N_0-2)-\ind(w).
\end{equation}
This identity is the numerical bridge from the original stable-window data to the
normalized weight introduced in the next section.

\begin{example}[A dual affine output]\label{ex:prelim-dars}
For the affine permutation in Example~\ref{ex:prelim-window}, Huang--Zhang's growth
diagram gives
\[
 \bar P=\yttab{\ol1&\ol3\\ \ol2\\ \ol4},
 \qquad
 \bar Q=\yttab{\ol1&\ol2\\ \ol3\\ \ol4},
 \qquad
 \lambda=(6,6,5),
 \qquad
 N_0=7.
\]
The common tabloid shape is \(\mu=(2,1,1)\).  Since \(\ind(w)=3\), the size check is
\[
              |\lambda|=17=4(7-2)-3,
\]
as required by \eqref{eq:dual-size-prelim}.  For comparison, a finite permutation
embedded in \(\Sn^{(0)}\) has \(N_0=2\) and \(\lambda=\varnothing\); thus the
nonempty partition above measures genuinely affine stable-window displacement.
\end{example}

\section{The normalized dual weight \texorpdfstring{\(\beta\)}{beta}}\label{sec:beta}

This section isolates the linear drift in the stable partitions and proves that the new
normalization is completely consistent with the original \((\lambda,N_0)\)-data.

\begin{lemma}[Evolution of stable partitions]\label{lem:stable-evolution}
For every \(M\ge N_0\),
\begin{equation}\label{eq:stable-evolution}
                   \lambda[M+1]=\lambda[M]+\mu.
\end{equation}
Consequently,
\begin{equation}\label{eq:stable-translation}
                   \lambda[M]=\lambda+(M-N_0)\mu.
\end{equation}
\end{lemma}

\begin{proof}
Once a window is stable, its east and west color sequences agree, and the same is true
for the north and south sequences.  Translating the window by one period therefore adds
one box to the stable partition for each occurrence of a row color.  Color \(i\) occurs
\(\mu_i\) times, so the increment is the shape vector \(\mu\).  Iteration gives
\eqref{eq:stable-translation}.
\end{proof}

\begin{definition}[Normalized dual weight]\label{def:beta}
For \(w\in\Sn\) with original dual data
\(\RSd(w)=(\bar P,\bar Q,\lambda,N_0)\), define
\begin{equation}\label{eq:beta-first}
                         \beta(w):=\mu(N_0-2)-\lambda.
\end{equation}
More generally, if \(M\ge N_0\) is any stable-window index, set
\begin{equation}\label{eq:beta-M}
                         \beta[M]:=\mu(M-2)-\lambda[M].
\end{equation}
\end{definition}

\begin{theorem}[Consistency with \texorpdfstring{\((\lambda,N_0)\)}{(lambda,N0)}]\label{thm:beta-consistency}
The normalized weight has the following properties.
\begin{enumerate}[label=(\roman*),leftmargin=2.2em]
\item \(\beta[M]=\beta(w)\) for every stable window \(M\ge N_0\).
\item Replacing the original pair \((\lambda,N_0)\) by \((\beta,N_0)\) loses no
information: the inverse formula is
\begin{equation}\label{eq:lambda-from-beta}
                         \lambda=\mu(N_0-2)-\beta.
\end{equation}
\item For any later stable window, the original stable partition is recovered by
\begin{equation}\label{eq:lambdaM-from-beta}
                         \lambda[M]=\mu(M-2)-\beta.
\end{equation}
Thus the data \((\bar P,\bar Q,\lambda,N_0)\) and
\((\bar P,\bar Q,\beta,N_0)\) are equivalent descriptions of the same dual affine
correspondence.
\end{enumerate}
\end{theorem}

\begin{proof}
By Lemma~\ref{lem:stable-evolution},
\[
 \beta[M]
 =\mu(M-2)-\lambda-(M-N_0)\mu
 =\mu(N_0-2)-\lambda
 =\beta(w).
\]
This proves (i).  Formula~\eqref{eq:lambda-from-beta} is the rearrangement of
\eqref{eq:beta-first}, proving (ii), and the same rearrangement of
\eqref{eq:beta-M} proves (iii).  Minimality of \(N_0\) is retained explicitly, so the
first stable window is not forgotten.
\end{proof}

\begin{proposition}[Index normalization]\label{prop:beta-index}
The coordinate sum of \(\beta\) is the index of the affine permutation:
\begin{equation}\label{eq:beta-index}
                         \sum_{i=1}^{\ell}\beta_i(w)=\ind(w).
\end{equation}
In particular, \(\beta\) and the AMBC weight \(\rho\) have the same coordinate sum.
\end{proposition}

\begin{proof}
The first-stable-window size formula \eqref{eq:dual-size-prelim} gives
\(|\lambda|=n(N_0-2)-\ind(w)\).
Since \(|\mu|=n\), summing \eqref{eq:beta-first} gives
\(\sum_i\beta_i=n(N_0-2)-|\lambda|=\ind(w)\).  Compare with
\eqref{eq:rho-index}.
\end{proof}

\begin{remark}
The integer \(N_0\) remains useful because it records where stability first begins.
The role of \(\beta\) is different: it removes the predictable translation
\(\lambda[M+1]-\lambda[M]=\mu\) and provides a weight that can be compared directly
with \(\rho\).
\end{remark}

\section{Transformations of the dual affine correspondence}\label{sec:operations}

We now collect three natural operations: inversion, affine Knuth moves, and changes of
window.  Throughout, \(\beta\) is the normalized weight from
Definition~\ref{def:beta}.

\subsection{Inversion}

\begin{theorem}[Inversion of the tabloid data]\label{thm:inverse-tabloids}
Let \(\RSd(w)=(\bar P,\bar Q,\lambda,N_0)\).  Then
\begin{equation}\label{eq:inverse-tabloids}
 \bar P(w^{-1})=\E(\bar Q(w)),
 \qquad
 \bar Q(w^{-1})=\E(\bar P(w)).
\end{equation}
For the reflected inverse \(\iota(w)=r(w^{-1})\), one has the cleaner antidiagonal
symmetry
\begin{equation}\label{eq:reflected-inverse-tabloids}
 \bar P(\iota(w))=\bar Q(w),
 \qquad
 \bar Q(\iota(w))=\bar P(w).
\end{equation}
\end{theorem}

\begin{proof}
Write \(\AMBC(w)=(P,Q,\rho)\).  By \eqref{eq:AMBC-inverse}, the two AMBC tabloids for
\(w^{-1}\) are \((Q,P)\).  The comparison \eqref{eq:tabloid-comparison} and the
involutivity of \(\E\) give
\[
 \bar P(w^{-1})=Q=\E(\bar Q(w)),
 \qquad
 \bar Q(w^{-1})=\E(P)=\E(\bar P(w)).
\]
The second statement is the antidiagonal-reflection symmetry of the dual growth diagram
\cite[Proposition~5.7]{HZ26}; applying it twice gives both identities in
\eqref{eq:reflected-inverse-tabloids}.
\end{proof}

Unlike the tabloid part, the inverse weight is not obtained by a bare sign: the vector
\(-\rho\) must first be moved to the dominant representative for the reversed tabloid
pair.  The exact formula for \(\beta(w^{-1})\) is Corollary~\ref{cor:inverse-beta}, after
the correction potential has been constructed.

\begin{lemma}[Stable cut crossing]\label{lem:stable-cut}
Let \(M\) be late enough that the standard windows used for \(w\), \(\Rop w\), and
\(\Lop w\) are stable.  Moving the domain cut through one period leaves the east
boundary tabloid fixed, replaces the north boundary tabloid \(\bar Q\) by
\(\sigma\bar Q\), and changes the northeast-corner partition by
\[
 \lambda_{\Rop w}[M]=\lambda_w[M]-\e_{\row_{\bar Q}(\ol n)}.
\]
Moving the image cut through one period gives, symmetrically,
\[
 \bar P\longmapsto\sigma\bar P,\qquad
 \bar Q\longmapsto\bar Q,\qquad
 \lambda_{\Lop w}[M]=\lambda_w[M]-\e_{\row_{\bar P}(\ol n)}.
\]
Consequently,
\[
 \beta(\Rop w)-\beta(w)=\e_{\row_{\bar Q}(\ol n)},\qquad
 \beta(\Lop w)-\beta(w)=\e_{\row_{\bar P}(\ol n)}.
\]
\end{lemma}

\begin{proof}
Choose a common stable region above all marked tiles.  Translating the vertical cut by
one column identifies every tile of the two diagrams except the single periodic strip
through which the representative with north-boundary residue \(\ol n\) passes.  On the
identified boundaries the colors are unchanged, while their residue labels increase by
one; hence the north tabloid becomes \(\sigma\bar Q\) and the east tabloid is fixed.

It remains to compare the corner partitions.  Along the exceptional strip, apply the
dual local rule successively from its southwest end to its northeast end.  At a cross
tile the two labels are merely carried to the opposite edges.  At a bump tile stability
forces the two incident shadow lines to have the same color, so decreasing the incoming
label by one decreases the outgoing label of that same color by one and does not affect
any other color.  Thus the discrepancy propagates along the unique stable shadow line
that crosses the cut.  Its color is \(r=\row_{\bar Q}(\ol n)\), and at the northeast
corner exactly one box is removed from row \(r\).  This proves the first partition
identity.  The horizontal statement is the transpose of the same strip calculation.
Finally use \(\beta[M]=\mu(M-2)-\lambda[M]\) in the common stable window.
\end{proof}

\subsection{Affine Knuth moves}

\begin{theorem}[Dual affine Knuth rule]\label{thm:dual-knuth}
Suppose \(w\) and \(w'\) differ by a right affine Knuth move.  Orient the induced tabloid
move from \(\bar Q=\bar Q(w)\) to \(\bar Q'=\bar Q(w')\).  Then
\[
                         \bar P(w')=\bar P(w),
\]
and \(\bar Q'\) is obtained from \(\bar Q\) by the induced affine Knuth move.  If the
exchanged entries are \(\ol i,\ol{i+1}\) with \(i\ne n\), then
\begin{equation}\label{eq:beta-noncyclic-knuth}
                         \beta(w')=\beta(w).
\end{equation}
If the move exchanges \(\ol n\) and \(\ol1\), then
\begin{equation}\label{eq:beta-cyclic-knuth}
 \beta(w')-\beta(w)
 =\e_{\row_{\bar Q}(\ol n)}-\e_{\row_{\bar Q}(\ol1)}.
\end{equation}
\end{theorem}

\begin{proof}
The local dual growth rules are the ordinary dual Knuth rules in every finite region
\cite{Fom95}.  For a noncyclic move choose the fundamental cut away from the two
interchanged positions.  The three possible relative orders in the affected rank-two
interval give the usual dual Knuth diamond: the east boundary and the northeast-corner
partition are unchanged, while the two north-boundary residues are interchanged.  This
proves \eqref{eq:beta-noncyclic-knuth} and fixes \(\bar P\).

Suppose now that the move interchanges \(\ol n\) and \(\ol1\).  After applying \(\Rop\)
to both endpoints, the same move is noncyclic, and hence
\(\beta(\Rop w')=\beta(\Rop w)\).  Lemma~\ref{lem:stable-cut} gives
\[
 \beta(w')+\e_{\row_{\bar Q'}(\ol n)}
 =\beta(w)+\e_{\row_{\bar Q}(\ol n)}.
\]
Since \(\bar Q'\) is obtained by interchanging \(\ol n\) and \(\ol1\), one has
\(\row_{\bar Q'}(\ol n)=\row_{\bar Q}(\ol1)\).  Rearranging proves
\eqref{eq:beta-cyclic-knuth}.
\end{proof}

For comparison, the AMBC rule \cite[Theorem~3.11]{CLP18} is
\begin{equation}\label{eq:rho-knuth}
 \rho(w')-\rho(w)=
 \begin{cases}
  0,&i\ne n,\\[1.5mm]
  \e_{\row_Q(\ol1)}-\e_{\row_Q(\ol n)},&i=n,
 \end{cases}
\end{equation}
where \(Q=Q(w)\).  Inversion turns the theorem into the corresponding left affine
Knuth rule: the recording tabloid is fixed and the insertion tabloid undergoes the
induced move.

\begin{example}[Noncyclic and cyclic Knuth edges]\label{ex:knuth}
Consider
\[
 w=[1,4,6,2,5,3].
\]
The AMBC calculation in \cite[Examples~3.12--3.13]{CLP18} gives
\[
 P=Q=\yttab{\ol1&\ol2&\ol3\\ \ol4&\ol5\\ \ol6},
 \qquad \rho=(0,0,0).
\]
The noncyclic Knuth move
\[
 w'=[1,4,2,6,5,3]
\]
fixes \(P\), replaces \(Q\) by
\[
 Q'=\yttab{\ol1&\ol2&\ol4\\ \ol3&\ol5\\ \ol6},
\]
and leaves \(\rho\) unchanged.
Correspondingly, \(\bar P\) is fixed, \(\bar Q=\E(Q)\) changes to \(\E(Q')\), and
\(\beta(w')=\beta(w)\).

The cyclic Knuth move gives
\[
 w''=[-3,4,6,2,5,7],
 \qquad
 Q''=\yttab{\ol2&\ol3&\ol6\\ \ol4&\ol5\\ \ol1}.
\]
Here \(\ol6\) is in the third row of \(Q\) and \(\ol1\) is in the first, so
\[
 \rho(w'')-\rho(w)=\e_1-\e_3=(1,0,-1).
\]
On the dual side the same edge has
\[
 \beta(w'')-\beta(w)
 =\e_{\row_{\E(Q)}(\ol6)}-\e_{\row_{\E(Q)}(\ol1)}.
\]
This displays the two local cocycles whose period comparison produces the
two-potential formula.
\end{example}

\subsection{Window shifts}

\begin{theorem}[Window transformations]\label{thm:one-step}
Let \(\RSd(w)=(\bar P,\bar Q,\lambda,N_0)\) and
\(\AMBC(w)=(P,Q,\rho)\).  Then
\begin{align}
 \bar P(\Rop w)&=\bar P,&
 \bar Q(\Rop w)&=\sigma\bar Q,&
 \beta(\Rop w)&=\beta+\e_{\row_{\bar Q}(\ol n)},
 \label{eq:R-dual}\\
 \bar P(\Lop w)&=\sigma\bar P,&
 \bar Q(\Lop w)&=\bar Q,&
 \beta(\Lop w)&=\beta+\e_{\row_{\bar P}(\ol n)}.
 \label{eq:L-dual}
\end{align}
For AMBC,
\begin{align}
 P(\Rop w)&=P,&
 Q(\Rop w)&=\sigma^{-1}Q,&
 \rho(\Rop w)&=\rho+\e_{\row_Q(\ol1)},
 \label{eq:R-ambc}\\
 P(\Lop w)&=\sigma P,&
 Q(\Lop w)&=Q,&
 \rho(\Lop w)&=\rho+\e_{\row_P(\ol n)}.
 \label{eq:L-ambc}
\end{align}
\end{theorem}

\begin{proof}
The two dual formulas are Lemma~\ref{lem:stable-cut}.  For AMBC, translating the domain
cut leaves the image residues fixed and replaces the recording tabloid by
\(\sigma^{-1}Q\).  Exactly the stream containing \(\ol1\) in its domain set gains one
unit of altitude, which gives \eqref{eq:R-ambc}.  Translating the image cut is the
transposed calculation: the recording tabloid is fixed, the insertion tabloid becomes
\(\sigma P\), and the stream containing \(\ol n\) in its image set gains one unit,
giving \eqref{eq:L-ambc}.  These statements also follow directly from the defining-data
description of a stream in \cite[Section~3.4]{CPY18}.  Finally,
\(\E\sigma^{-1}=\sigma\E\) makes the two recording-tabloid formulas compatible with
\(\bar Q=\E(Q)\).
\end{proof}

Thus the three requested windows have normalized dual data
\begin{align*}
 [w(1),\ldots,w(n)]
   &\longleftrightarrow (\bar P,\bar Q,\beta),\\
 [w(2),\ldots,w(n+1)]
   &\longleftrightarrow
   (\bar P,\sigma\bar Q,\beta+\e_{\row_{\bar Q}(\ol n)}),\\
 [w(1)+1,\ldots,w(n)+1]
   &\longleftrightarrow
   (\sigma\bar P,\bar Q,\beta+\e_{\row_{\bar P}(\ol n)}).
\end{align*}

\begin{corollary}[Iterated window transformations]\label{cor:iterated-shifts}
For \(0\le m\le n\),
\begin{align}
 \beta(\Rop^m w)&=\beta(w)+c_{\bar Q}(m),&
 \rho(\Rop^m w)&=\rho(w)+d_Q(m),
 \label{eq:iterate-R}\\
 \beta(\Lop^m w)&=\beta(w)+c_P(m),&
 \rho(\Lop^m w)&=\rho(w)+c_P(m).
 \label{eq:iterate-L}
\end{align}
In particular,
\begin{equation}\label{eq:central-shift}
 \beta(w+n)=\beta(w)+\mu,
 \qquad
 \rho(w+n)=\rho(w)+\mu.
\end{equation}
\end{corollary}

\begin{proof}
Apply Theorem~\ref{thm:one-step} successively and record the residue crossing the cut at
each step.  At \(m=n\), every residue has crossed once, and
\(c_U(n)=d_U(n)=\mu\).
\end{proof}

\begin{example}[A complete window-shift calculation]\label{ex:window-shift}
The dual growth diagram in \cite[Example~3.17]{HZ26} gives, for
\(w=[10,3,-3,12]\in\widetilde S_4\),
\begin{align*}
 \bar P&=\yttab{\ol1&\ol3\\ \ol2\\ \ol4},&
 \bar Q&=\yttab{\ol1&\ol2\\ \ol3\\ \ol4},\\
 \mu&=(2,1,1),&
 \lambda&=(6,6,5),\qquad N_0=7.
\end{align*}
Thus
\[
 \beta=5\mu-\lambda=(4,-1,0),
 \qquad \sum_i\beta_i=3=\ind(w).
\]
Since \(\ol4\) is in the third row of both \(\bar P\) and \(\bar Q\),
Theorem~\ref{thm:one-step} gives
\[
 \Rop w=[3,-3,12,14],\qquad
 (\bar P,\bar Q,\beta)(\Rop w)
 =\left(\bar P,\yttab{\ol2&\ol3\\ \ol4\\ \ol1},(4,-1,1)\right),
\]
and
\[
 \Lop w=[11,4,-2,13],\qquad
 (\bar P,\bar Q,\beta)(\Lop w)
 =\left(\yttab{\ol2&\ol4\\ \ol3\\ \ol1},\bar Q,(4,-1,1)\right).
\]
After four steps, both routes give
\(\beta(w+4)=\beta(w)+\mu=(6,0,1)\), illustrating
\(\Rop^4=\Lop^4\) at the level of the cocycle.
\end{example}

\subsection{Algebraic relations among the transformations}

Let \(\mathfrak I(w)=w^{-1}\).  The elementary transformations generate a small
dihedral-type symmetry around the central translation \(z(w)=w+n\).

\begin{proposition}[Transformation algebra]\label{prop:transformation-algebra}
On \(\Sn\), one has
\begin{align}
 \mathfrak I^2&=1,
 &\Rop\Lop&=\Lop\Rop,
 &\Rop^n&=\Lop^n=z,
 \label{eq:transformation-algebra-1}\\
 \mathfrak I\Rop&=\Lop^{-1}\mathfrak I,
 &\mathfrak I\Lop&=\Rop^{-1}\mathfrak I.
 \label{eq:transformation-algebra-2}
\end{align}
Inversion takes every right affine Knuth move to a left affine Knuth move and conversely.
Consequently, all formulas for left moves follow from the right-move formulas together
with Theorem~\ref{thm:inverse-tabloids} and Corollary~\ref{cor:inverse-beta}.
\end{proposition}

\begin{proof}
The first three identities follow from associativity, the definitions
\(\Rop w=w\circ\tau\), \(\Lop w=\tau\circ w\), and periodicity.  Moreover,
\[
 (\Rop w)^{-1}=\tau^{-1}\circ w^{-1}=\Lop^{-1}(w^{-1}),
 \qquad
 (\Lop w)^{-1}=w^{-1}\circ\tau^{-1}=\Rop^{-1}(w^{-1}),
\]
which proves \eqref{eq:transformation-algebra-2}.  Matrix transposition exchanges rows
and columns, so it exchanges right and left Knuth moves.
\end{proof}

These group identities impose nontrivial consistency relations on the tabloid and weight
formulas.  For example, the two factorizations of the central translation give the same
total increment
\[
 \sum_{a=1}^{n}\e_{\row_{\bar Q}(\ol a)}
 =\sum_{a=1}^{n}\e_{\row_{\bar P}(\ol a)}=\mu.
\]
Likewise, applying inversion to a domain shift converts the row contribution on the
recording side to the corresponding inverse value-shift contribution on the insertion
side.  The dominant-representative term in \eqref{eq:inverse-beta-final} is precisely
what is required for this compatibility.

\subsection{The cocycle viewpoint}\label{subsec:cocycle}

Let \(\mathcal G_n\) be the directed graph whose vertices are extended affine
permutations and whose edges are right or left affine Knuth moves and the four window
edges \(\Rop^{\pm1},\Lop^{\pm1}\).  Define the vector-valued edge function
\begin{equation}\label{eq:beta-cocycle}
 c_\beta(w\to w'):=\beta(w')-\beta(w)\in\Z^\ell.
\end{equation}

\begin{proposition}[The transformation cocycle]\label{prop:beta-cocycle}
The function \(c_\beta\) is an integral \(1\)-cocycle.  More explicitly,
\begin{enumerate}[label=(\roman*),leftmargin=2.2em]
\item reversing an edge negates its value;
\item the value on a concatenated path is the sum of the edge values;
\item the sum is zero around every closed path in \(\mathcal G_n\);
\item on right-move and window edges its values are exactly
\begin{equation}\label{eq:beta-cocycle-values}
 \begin{array}{c|c}
 \text{edge}&c_\beta\\ \hline
 \text{noncyclic Knuth}&0\\
 \text{cyclic Knuth}&
   \e_{\row_{\bar Q}(\ol n)}-\e_{\row_{\bar Q}(\ol1)}\\
 w\to\Rop w&\e_{\row_{\bar Q}(\ol n)}\\
 w\to\Lop w&\e_{\row_{\bar P}(\ol n)}.
 \end{array}
\end{equation}
\end{enumerate}
The same construction gives a cocycle \(c_\rho\), and the difference
\(c_\beta-c_\rho\) is again an integral \(1\)-cocycle.
\end{proposition}

\begin{proof}
The first three assertions telescope because \(c_\beta\) is the differential of the
vertex function \(\beta\) on the lifted permutation graph.  The displayed values are
Theorems~\ref{thm:dual-knuth} and \ref{thm:one-step}.  Subtraction gives the last
statement.
\end{proof}

The distinction between exactness and monodromy appears only after projecting away the
weight coordinate.  On \(\mathcal G_n\), the cocycles are exact.  On the tabloid graph,
a closed tabloid path can lift to a path whose endpoints have different weights; the
period of the projected cocycle is then the monodromy.  Taking suitable block-prefix
sums of these vector cocycles produces the scalar cochains used in
Section~\ref{sec:comparison}.

\section{The relation between \texorpdfstring{\(\rho\)}{rho} and
\texorpdfstring{\(\beta\)}{beta}}\label{sec:comparison}

This section contains the two principal comparison theorems.  First we show that
\(\beta-\rho\) is independent of the altitude in a fixed tabloid fiber.  We then prove
that the monodromies agree and use this equality to construct the two-potential formula.

\subsection{Altitude independence and the correction vector}

Decompose the row indices into maximal equal-part blocks
\begin{equation}\label{eq:blocks}
 \{1,\ldots,\ell\}=B_1\sqcup\cdots\sqcup B_k,
 \qquad B_s=[a_s,b_s],
 \qquad q_s=|B_s|.
\end{equation}
For each block, write
\begin{equation}\label{eq:block-data}
 d_s:=\mu_{a_s},\qquad \tau_s:=\sum_{j>b_s}\mu_j,
 \qquad b_0:=0,\qquad \tau_0:=n.
\end{equation}
For \(a\in B_s\), set
\begin{equation}\label{eq:packet-vector}
                         v_{s,a}:=\sum_{j=a}^{b_s}\e_j.
\end{equation}

\begin{lemma}[Packet shift in a stable window]\label{lem:packet-stable-window}
Let \((P,Q,\rho)\) and \((P,Q,\rho+v_{s,a})\) be dominant AMBC data, and let
\(w\) and \(w^+\) be their inverse AMBC images.  There is an index \(M_1\) such that,
for every \(M\ge M_1\), the corresponding stable corner partitions satisfy
\begin{equation}\label{eq:packet-lambda}
             \lambda_{w^+}[M]=\lambda_w[M]-v_{s,a}.
\end{equation}
The stable color multiplicities and both stable boundary tabloids are the same for
\(w\) and \(w^+\).
\end{lemma}

\begin{proof}
We compare the two backward AMBC constructions row by row.  A row of the AMBC data
specifies a stream by its domain residues, image residues, and altitude
\cite[Definitions~3.20--3.25]{CPY18}.  Increasing its altitude by one translates its
proper stream numbering by one period.  The backward-numbering algorithm is equivariant
under this translation: its zigzags have the same residue sets and their back
corner-posts are shifted through one fundamental cut.

The shifted rows are the suffix \(a,a+1,\ldots,b_s\) of one equal-density block.  At
each internal interface both adjacent streams are translated together.  At the only
external interface, between rows \(a-1\) and \(a\) when \(a>a_s\), dominance says that
the second stream remains on the same side of the concurrent position.  Hence no two
backward zigzag families are interchanged.  Induction over the backward steps shows that
the two periodic ball diagrams have the same stream residue data and that, for every
\(j\in[a,b_s]\), their color-\(j\) families differ by one translated fundamental strip.

Choose \(M_1\) above every marked tile and beyond the stabilization indices of both
diagrams.  In the \(M\)-th windows identify the complements of those strips.  Along a
shifted strip, the dual local rule transports a discrepancy of one in the color-\(j\)
label to the northeast corner.  Cross tiles transport the label unchanged; at a bump
tile stability forces the incident lines to have the same color, so the discrepancy
cannot pass to another family.  Thus precisely one box is removed from row \(j\) of the
corner partition.  Summing over the shifted colors gives \eqref{eq:packet-lambda}.

The domain and image residue sets of every stream were unchanged, so the AMBC tabloids
remain \((P,Q)\).  Equation~\eqref{eq:tabloid-comparison} then gives the same stable
boundary tabloids \((P,\E(Q))\), and the stable multiplicities remain \(\mu\).
\end{proof}

\begin{proposition}[Altitude packet deformation]\label{prop:altitude}
Let \(w=\AMBC^{-1}(P,Q,\rho)\).  If \(\rho+v_{s,a}\) is dominant for \((P,Q)\) and
\(w^+=\AMBC^{-1}(P,Q,\rho+v_{s,a})\), then the dual tabloids are unchanged and
\begin{equation}\label{eq:beta-packet}
                         \beta(w^+)=\beta(w)+v_{s,a}.
\end{equation}
The same statement holds for every integral combination of packet vectors that remains
in the dominant region.
\end{proposition}

\begin{proof}
Take a common stable index \(M\) as in Lemma~\ref{lem:packet-stable-window}.  Then
\[
 \beta(w^+)
 =\mu(M-2)-\lambda_{w^+}[M]
 =\mu(M-2)-\lambda_w[M]+v_{s,a}
 =\beta(w)+v_{s,a}.
\]
The assertion about the dual tabloids is the last statement of that lemma.  Iterating
the one-packet result proves the assertion for an integral combination along any path
that stays dominant; the connectivity argument in the next proof shows that such a path
may always be chosen when both endpoints are dominant.
\end{proof}

\begin{corollary}[Tabloid correction]\label{cor:kappa-independent}
For fixed \((P,Q)\), the difference \(\beta-\rho\) is independent of the dominant weight
\(\rho\).  Hence
\begin{equation}\label{eq:kappa}
                         \kappa(P,Q):=\beta-\rho
\end{equation}
is a well-defined tabloid correction, and
\begin{equation}\label{eq:kappa-sum}
                         \sum_{i=1}^{\ell}\kappa_i(P,Q)=0.
\end{equation}
\end{corollary}

\begin{proof}
For one equal-row block \(B_s\), write
\[
 \delta_i(\rho):=\rho_{i+1}-\rho_i-\lch_i(P)+\lch_i(Q)
 \qquad(a_s\le i<b_s).
\]
Dominance is exactly \(\delta_i(\rho)\ge0\).  Adding \(v_{s,i+1}\) increases
\(\delta_i\) by one and fixes all other block slacks, whereas \(v_{s,a_s}\) translates
every coordinate of the block by one.  Given two dominant integer points, adjust each
slack to its target value, using an inverse move only while that slack is positive, and
then translate the whole block.  Repeating this independently in every block gives a
dominant packet path between the two points.

Proposition~\ref{prop:altitude} says that \(\beta\) and \(\rho\) receive the same
increment on every edge of this path.  Their difference is therefore constant.
Equation~\eqref{eq:kappa-sum} follows from \eqref{eq:beta-index} and
\eqref{eq:rho-index}.
\end{proof}

\begin{lemma}[Fiber gauge identity]\label{lem:fiber-gauge}
Fix an insertion tabloid \(P\).  Let \(e:Q\to Q'\) be either a right affine Knuth
edge or the rotation edge induced by \(w\mapsto\Rop w\).  If
\(c_\rho(e)\) and \(c_\beta(e)\) denote the corresponding AMBC and normalized-dual
weight increments, then
\begin{equation}\label{eq:fiber-gauge}
 c_\beta(e)-c_\rho(e)=\kappa(P,Q')-\kappa(P,Q).
\end{equation}
Consequently, the difference cocycle has zero period on every closed path in the
recording-fiber graph, including paths containing rotation edges.
\end{lemma}

\begin{proof}
At the initial and terminal vertices, Corollary~\ref{cor:kappa-independent} gives
\(\beta=\rho+\kappa(P,Q)\) and
\(\beta'=\rho'+\kappa(P,Q')\).  Subtraction is exactly
\eqref{eq:fiber-gauge}.  Summing it along a path telescopes, and on a closed path the two
endpoint values of \(\kappa\) agree.
\end{proof}

\subsection{Common monodromy}

Let \(\A_\mu\) be the affine Knuth graph on tabloids of shape \(\mu\).  A path lifts,
for a fixed insertion tabloid, to a sequence of affine permutations.  If the path closes
in the tabloid graph, the endpoint weight may be translated from the initial weight; the
translation is the monodromy of the lifted path.

\begin{theorem}[Common monodromy]\label{thm:monodromy}
After identifying the recording tabloids by \(\bar Q=\E(Q)\), the normalized dual weight
\(\beta\) and the AMBC weight \(\rho\) have the same monodromy on every closed affine
Knuth path.  Their common monodromy lattice is
\begin{equation}\label{eq:monodromy-lattice}
 G_\mu=
 \left\{x\in\Z^\ell:
   \sum_{i=1}^{\ell}x_i=0,
   \quad x_i=x_{i+1}\text{ whenever }\mu_i=\mu_{i+1}
 \right\}.
\end{equation}
If \(k\) is the number of distinct row lengths of \(\mu\), then
\(\operatorname{rank}G_\mu=k-1\).
\end{theorem}

\begin{proof}
Suppose a lifted path returns to the same \((P,Q)\) and changes the AMBC weight from
\(\rho\) to \(\rho+x\).  Corollary~\ref{cor:kappa-independent} gives
\[
 \beta_{\mathrm{end}}
   =(\rho+x)+\kappa(P,Q)
   =\beta_{\mathrm{start}}+x.
\]
Thus the two monodromies agree pointwise, not merely as abstract lattices.
Chmutov--Lewis--Pylyavskyy computed the AMBC monodromy lattice as
\eqref{eq:monodromy-lattice} \cite[Theorem~7.28]{CLP18}; the same description therefore
holds for \(\beta\).  A vector in this lattice is constant on each of the \(k\) blocks,
and its \(k\) block values satisfy the single relation
\(\sum_s q_sx_s=0\); hence the rank is \(k-1\).
\end{proof}

At a block endpoint \(b_s\), define the prefix map
\begin{equation}\label{eq:prefix-map}
                         \pi_s(x):=\sum_{i=1}^{b_s}x_i.
\end{equation}
Write \(\pi=(\pi_1,\ldots,\pi_k)\).  Because an element of \(G_\mu\) is constant on
each block, its block-prefix values obey
\begin{equation}\label{eq:prefix-lattice}
 \pi(G_\mu)=
 \left\{(y_1,\ldots,y_k)\in\Z^k:
 y_0=y_k=0,\quad y_s-y_{s-1}\in q_s\Z\right\}.
\end{equation}

\subsection{The two-potential formula}

The equality of monodromies may be seen directly on a generator.  For
\(\mu=(2,1)\), the lattice is
\[
 G_\mu=\{(m,-m):m\in\Z\}.
\]
Take \(x=(1,-1)\) and choose a closed tabloid path \(\gamma\) realizing \(x\).
The two lifted paths in Figure~\ref{fig:monodromy-equality} return to the same recording
tabloid but translate their weights by the same vector.

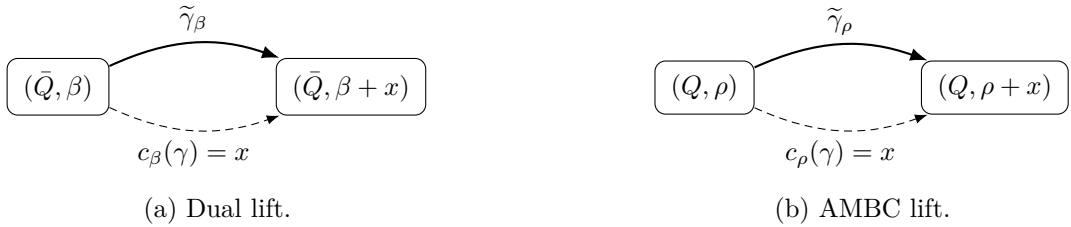
\begin{figure}[htbp]
\centering
\begin{minipage}[t]{0.47\textwidth}
\centering
\begin{tikzpicture}[
  >=Latex,
  every node/.style={font=\small},
  state/.style={draw,rounded corners,inner xsep=6pt,inner ysep=5pt}
]
 \node[state] (b0) {$(\bar Q,\beta)$};
 \node[state,right=22mm of b0] (b1) {$(\bar Q,\beta+x)$};
 \draw[->,thick,bend left=22]
   (b0) to node[above] {$\widetilde\gamma_\beta$} (b1);
 \draw[->,densely dashed,bend right=22]
   (b0) to node[below] {$c_\beta(\gamma)=x$} (b1);
\end{tikzpicture}

\smallskip
\small (a) Dual lift.
\end{minipage}\hfill
\begin{minipage}[t]{0.47\textwidth}
\centering
\begin{tikzpicture}[
  >=Latex,
  every node/.style={font=\small},
  state/.style={draw,rounded corners,inner xsep=6pt,inner ysep=5pt}
]
 \node[state] (r0) {$(Q,\rho)$};
 \node[state,right=22mm of r0] (r1) {$(Q,\rho+x)$};
 \draw[->,thick,bend left=22]
   (r0) to node[above] {$\widetilde\gamma_\rho$} (r1);
 \draw[->,densely dashed,bend right=22]
   (r0) to node[below] {$c_\rho(\gamma)=x$} (r1);
\end{tikzpicture}

\smallskip
\small (b) AMBC lift.
\end{minipage}
\caption{Two lifts of the same closed tabloid path for
\(\mu=(2,1)\).  In both diagrams \(x=(1,-1)\), so
\(\operatorname{Mon}_\beta(\gamma)=x=\operatorname{Mon}_\rho(\gamma)\).}
\label{fig:monodromy-equality}
\end{figure}

Enlarge \(\A_\mu\) by adding the rotation edge
\(Q\to\sigma^{-1}Q\), and call the resulting graph \(\widehat\A_\mu\).
This enlarged graph is connected: affine Knuth components are classified by the charge
congruence, and the shift \(Q\mapsto\sigma^{-1}Q\) acts transitively on those components
\cite[Theorem~8.6 and Remark~8.8]{CLP18}.  For a block
endpoint \(b_s\), let
\begin{equation}\label{eq:chi}
 \chi_s(U;\ol a):=
 \begin{cases}
   1,&\row_U(\ol a)\le b_s,\\
   0,&\row_U(\ol a)>b_s.
 \end{cases}
\end{equation}
We define two integral edge cochains on \(\widehat\A_\mu\).  For an edge beginning at
\(Q\), the \emph{dual cochain} is
\begin{equation}\label{eq:omega-beta}
 \omega_s^\beta=
 \begin{cases}
  0,&\text{noncyclic Knuth edge},\\
  \chi_s(\E(Q);\ol n)-\chi_s(\E(Q);\ol1),&\text{cyclic Knuth edge},\\
  \chi_s(\E(Q);\ol n),&Q\to\sigma^{-1}Q,
 \end{cases}
\end{equation}
and the \emph{AMBC cochain} is
\begin{equation}\label{eq:omega-rho}
 \omega_s^\rho=
 \begin{cases}
  0,&\text{noncyclic Knuth edge},\\
  \chi_s(Q;\ol1)-\chi_s(Q;\ol n),&\text{cyclic Knuth edge},\\
  \chi_s(Q;\ol1),&Q\to\sigma^{-1}Q.
 \end{cases}
\end{equation}
Reverse orientations receive negative weights.  Formulas
\eqref{eq:beta-cyclic-knuth}, \eqref{eq:rho-knuth}, \eqref{eq:R-dual}, and
\eqref{eq:R-ambc} say that these are precisely the changes of the first \(b_s\)
coordinates of \(\beta\) and \(\rho\), respectively.

Fix the reverse row superstandard tabloid \(Q_*:=\RRSS(\mu,\ol1)\).  For an oriented path
\(\gamma:Q_*\rightsquigarrow Q\), define two path potentials
\begin{equation}\label{eq:path-potentials}
 \mathcal F_s^\beta(\gamma):=\int_\gamma\omega_s^\beta,
 \qquad
 \mathcal F_s^\rho(\gamma):=\int_\gamma\omega_s^\rho.
\end{equation}
Individually, these may depend on the homotopy class of \(\gamma\), because each records
a nontrivial monodromy.

\begin{lemma}[Base-point calibration]\label{lem:base-calibration}
For every insertion tabloid \(P\) of shape \(\mu\), the correction over
\(Q_*=\RRSS(\mu,\ol1)\) is independent of \(P\) and is given by
\begin{equation}\label{eq:kappa-star}
 \kappa_i(P,Q_*)=\kappa_i^*(\mu)
 :=(i-1)\mu_i-\sum_{j>i}\mu_j.
\end{equation}
Consequently, at every block endpoint \(b_s\),
\begin{equation}\label{eq:base-prefix}
 \sum_{i=1}^{b_s}\kappa_i(P,Q_*)=-b_s\sum_{j>b_s}\mu_j.
\end{equation}
\end{lemma}

\begin{proof}
By Corollary~\ref{cor:kappa-independent}, we may choose a convenient dominant weight.
Take \(\rho_i=H(i-1)\), where \(H\) is larger than every offset constant occurring for
\((P,Q_*)\).  In backward AMBC the stream families are then separated by more than one
fundamental window.  Consequently their backward zigzags are read, one family at a
time, by semi-infinite ordinary insertion; no zigzag belonging to one density can meet
a zigzag belonging to another density before the stable boundary.

We record the resulting boundary count.  The domain residues of \(Q_*\) occur in
consecutive blocks, beginning with row \(\ell\) and proceeding from bottom to top.
For a sufficiently late standard window \(M\), color \(i\) therefore contributes
\begin{equation}\label{eq:calibration-lambda-count}
 \lambda_i[M]
 =\mu_i(M-2)-\rho_i-(i-1)\mu_i+\sum_{j>i}\mu_j.
\end{equation}
Indeed, translating the color-\(i\) stream through one altitude removes one occurrence
from the northeast boundary count; passing each of the preceding \(i-1\) row blocks
removes \(\mu_i\) occurrences, while the initial bottom-to-top displacement contributes
one occurrence for every residue in the later rows.  This is also obtained inductively:
the bottom row starts with residue \(\ol1\), and moving from row \(i+1\) to row \(i\)
changes the starting boundary position by \(\mu_{i+1}\); the displayed expression is
the resulting telescoping count.  The image residues in \(P\) change the order of the
crossings inside one color family but not their number, so the count is independent of
\(P\).

Subtracting \eqref{eq:calibration-lambda-count} from \(\mu_i(M-2)\) gives
\[
 \beta_i=\rho_i+(i-1)\mu_i-\sum_{j>i}\mu_j,
\]
which proves \eqref{eq:kappa-star}.  To obtain \eqref{eq:base-prefix}, sum over
\(i\le b_s\).  Because \(b_s\) is the end of an equal-row block, the terms involving
\(\mu_j\) with \(j\le b_s\) cancel pairwise, while each \(\mu_j\) with \(j>b_s\)
occurs \(b_s\) times with a minus sign.
\end{proof}

\begin{theorem}[Two-potential formula]\label{thm:two-potential}
The difference of the two path potentials is independent of the chosen path from
\(Q_*\) to \(Q\).  Writing
\begin{equation}\label{eq:xch}
 \xch_s(Q):=
 \mathcal F_s^\beta(\gamma)-\mathcal F_s^\rho(\gamma),
 \qquad \xch_s(Q_*)=0,
\end{equation}
one has, for every insertion tabloid \(P\),
\begin{equation}\label{eq:two-potential-formula}
 \sum_{i=1}^{b_s}\bigl(\beta_i-\rho_i\bigr)
 =\mathcal F_s^\beta(\gamma)-\mathcal F_s^\rho(\gamma)
  -b_s\sum_{j>b_s}\mu_j.
\end{equation}
Equivalently,
\begin{equation}\label{eq:A-xch}
 A_s(P,Q):=\sum_{i=1}^{b_s}\kappa_i(P,Q)
 =\xch_s(Q)-b_s\sum_{j>b_s}\mu_j.
\end{equation}
In particular, the block prefix is independent of \(P\).
\end{theorem}

\begin{proof}
Apply the prefix map \(\pi_s\) to the fiber gauge identity
\eqref{eq:fiber-gauge}.  By \eqref{eq:omega-beta} and \eqref{eq:omega-rho}, on every
edge \(e:Q\to Q'\) of \(\widehat\A_\mu\) this gives
\[
 \omega_s^\beta(e)-\omega_s^\rho(e)
   =A_s(P,Q')-A_s(P,Q).
\]
Thus the difference cochain is exact, including on rotation edges.  In particular, its
integral has zero period around every closed path, so
\(\mathcal F_s^\beta(\gamma)-\mathcal F_s^\rho(\gamma)\) depends only on the endpoint
\(Q\).  Since \(\widehat\A_\mu\) is connected, this defines \(\xch_s\) on every tabloid.

Integrating the displayed equality from \(Q_*\) to \(Q\) yields
\[
 \xch_s(Q)=A_s(P,Q)-A_s(P,Q_*).
\]
Lemma~\ref{lem:base-calibration} gives
\(A_s(P,Q_*)=-b_s\sum_{j>b_s}\mu_j\), independently of \(P\).  Substitution proves
\eqref{eq:two-potential-formula} and \eqref{eq:A-xch}; the same formula also proves the
asserted independence from \(P\).
\end{proof}

The theorem explains the term \emph{two-potential}: neither potential is an ordinary
function on the tabloid graph when monodromy is present, but their difference is an
ordinary integer-valued statistic because the periods coincide.

For the endpoint conventions used below, set
\begin{equation}\label{eq:endpoint-conventions}
 A_0=0,\qquad \xch_0=0.
\end{equation}
Since \(b_k=\ell\) and \(\sum_i\kappa_i=0\), one also has
\(A_k=\xch_k=0\).

\subsection{Local gaps and coordinate reconstruction}

Suppose \(\mu_i=\mu_{i+1}\).  Let \(J_i^\vee(P,\bar Q)\) be the least possible value of
\(\beta_{i+1}-\beta_i\) among compatible stable dual data with tabloids \((P,\bar Q)\).
Equivalently, it is the displacement at which the two equal-density colored shadow-line
families first become saturated without crossing.

\begin{proposition}[Equality of local slacks]\label{prop:local-slack}
Let \(\bar Q=\E(Q)\) and \(\mu_i=\mu_{i+1}\).  The AMBC and dual dominance slacks agree:
\begin{equation}\label{eq:slack-equality}
 \rho_{i+1}-\rho_i-\lch_i(P)+\lch_i(Q)
 =\beta_{i+1}-\beta_i-J_i^\vee(P,\bar Q).
\end{equation}
Consequently,
\begin{equation}\label{eq:local-kappa}
 \kappa_{i+1}(P,Q)-\kappa_i(P,Q)=C_i(P,Q),
\end{equation}
where
\begin{equation}\label{eq:C}
 C_i(P,Q):=J_i^\vee(P,\E(Q))-\lch_i(P)+\lch_i(Q).
\end{equation}
\end{proposition}

\begin{proof}
Fix \((P,Q)\).  Corollary~\ref{cor:kappa-independent} gives, throughout this fiber,
\[
 \beta_{i+1}-\beta_i
 =\rho_{i+1}-\rho_i+\kappa_{i+1}(P,Q)-\kappa_i(P,Q).
\]
By the AMBC dominance condition \eqref{eq:AMBC-dominance}, the least possible value of
the first gap on the right is
\(\lch_i(P)-\lch_i(Q)\).  Conversely, equality is allowed in the dominant region, so
this lower bound is attained by a dominant integral weight.  The AMBC bijection and the
dual correspondence parametrize the same affine permutations in the fiber.  Hence the
least possible value of the left side is precisely
\(J_i^\vee(P,\E(Q))\).  Taking minima in the displayed identity gives
\[
 J_i^\vee(P,\E(Q))
 =\lch_i(P)-\lch_i(Q)
   +\kappa_{i+1}(P,Q)-\kappa_i(P,Q),
\]
which is \eqref{eq:local-kappa}.  Subtracting this minimum identity from the original
gap identity gives \eqref{eq:slack-equality}.
\end{proof}

For \(j\in B_s\), set
\begin{equation}\label{eq:Ssj}
 S_{s,j}:=\sum_{i=a_s}^{j-1}C_i(P,Q),
 \qquad S_{s,a_s}=0.
\end{equation}

\begin{theorem}[Local--global decomposition]\label{thm:local-global}
For every equal-row block \(B_s=[a_s,b_s]\), put
\begin{equation}\label{eq:Sbar}
 \overline S_s:=\frac1{q_s}\sum_{r=a_s}^{b_s}S_{s,r}.
\end{equation}
Then the restriction of \(\kappa=\beta-\rho\) to \(B_s\) has the canonical
decomposition
\begin{equation}\label{eq:local-global}
 \kappa_j=
 \underbrace{S_{s,j}-\overline S_s}_{\displaystyle\kappa^{\mathrm{loc}}_j}
 +
 \underbrace{\frac{A_s-A_{s-1}}{q_s}}_{
             \displaystyle\kappa^{\mathrm{glob}}_s},
 \qquad j\in B_s.
\end{equation}
Here the local part has sum zero on \(B_s\) and is determined by the consecutive
dual-gap/local-charge data \(C_i\), while the global part is constant on \(B_s\) and is
determined by the two potentials.  More explicitly,
\begin{equation}\label{eq:global-potential}
 \kappa^{\mathrm{glob}}_s
 =b_{s-1}d_s-\tau_s
  +\frac{\xch_s(Q)-\xch_{s-1}(Q)}{q_s}.
\end{equation}
This is the unique decomposition of \(\kappa|_{B_s}\) into a zero-sum vector and a
constant vector.
\end{theorem}

\begin{proof}
By Proposition~\ref{prop:local-slack},
\(\kappa_j=\kappa_{a_s}+S_{s,j}\) in the block.  Averaging over \(B_s\) shows that
the block average of \(\kappa\) is \((A_s-A_{s-1})/q_s\), and subtracting this average
gives \(S_{s,j}-\overline S_s\).  This proves \eqref{eq:local-global} and the asserted
uniqueness.  Finally, Theorem~\ref{thm:two-potential} gives
\[
 A_s-A_{s-1}
 =\xch_s-\xch_{s-1}-b_s\tau_s+b_{s-1}\tau_{s-1}.
\]
Since \(b_s=b_{s-1}+q_s\) and \(\tau_{s-1}=q_sd_s+\tau_s\), division by \(q_s\)
gives \eqref{eq:global-potential}.
\end{proof}

\begin{corollary}[Complete \texorpdfstring{\(\beta\)--\(\rho\)}{beta-rho} formula]\label{cor:complete}
Let \(\AMBC(w)=(P,Q,\rho)\).  With the block notation and endpoint conventions
above, for \(j\in B_s\) one has
\begin{equation}\label{eq:complete-formula}
 \begin{aligned}
 \beta_j={}&\rho_j+b_{s-1}d_s-\tau_s+S_{s,j}\\
 &+\frac{\xch_s(Q)-\xch_{s-1}(Q)
              -\displaystyle\sum_{r=a_s}^{b_s}S_{s,r}}{q_s}.
 \end{aligned}
\end{equation}
The numerator in the last term satisfies
\begin{equation}\label{eq:divisibility}
 \xch_s(Q)-\xch_{s-1}(Q)
 -\sum_{r=a_s}^{b_s}S_{s,r}\equiv0\pmod{q_s}.
\end{equation}
\end{corollary}

\begin{proof}
Insert \eqref{eq:Sbar} and \eqref{eq:global-potential} into the decomposition
\eqref{eq:local-global}, and then use \(\beta_j=\rho_j+\kappa_j\).  This gives
\eqref{eq:complete-formula}.  The congruence follows because the block-constant term in
\eqref{eq:local-global} gives
\[
 \frac{\xch_s-\xch_{s-1}-\sum_{r=a_s}^{b_s}S_{s,r}}{q_s}
 =\kappa_{a_s}-b_{s-1}d_s+\tau_s\in\Z.
\]
\end{proof}

\subsection{Special shapes}

The local--global decomposition separates the ways in which repeated and distinct row
lengths contribute.  We record three useful extremes.

\begin{corollary}[Rectangular shape]\label{cor:rectangle}
Suppose \(\mu=(d^\ell)\).  There is one equal-row block, the monodromy lattice is
trivial, and the global part of \(\kappa\) vanishes.  With
\[
 S_j:=\sum_{i=1}^{j-1}C_i(P,Q),\qquad S_1=0,
\]
one has
\begin{equation}\label{eq:rectangle-formula}
 \beta_j=\rho_j+S_j-\frac1\ell\sum_{r=1}^{\ell}S_r.
\end{equation}
Thus in rectangular shape the consecutive local slacks determine the complete
correction.
\end{corollary}

\begin{proof}
Here \(k=1\), \(A_0=A_1=0\), and \(G_\mu=0\).  Formula
\eqref{eq:local-global} immediately gives \eqref{eq:rectangle-formula}.
\end{proof}

\begin{corollary}[Pairwise distinct row lengths]\label{cor:distinct-parts}
Suppose \(\mu_1>\mu_2>\cdots>\mu_\ell\).  Every block is a singleton, so there are no
local slack terms.  With \(\xch_0=\xch_\ell=0\),
\begin{equation}\label{eq:distinct-formula}
 \beta_s-\rho_s
 =(s-1)\mu_s-\sum_{j>s}\mu_j
  +\xch_s(Q)-\xch_{s-1}(Q).
\end{equation}
Thus the two path potentials alone determine the correction.
\end{corollary}

\begin{proof}
Set \(q_s=1\), \(a_s=b_s=s\), and \(S_{s,s}=0\) in
Corollary~\ref{cor:complete}.
\end{proof}

\begin{corollary}[Exactly two row lengths]\label{cor:two-lengths}
Let \(\mu=(d^q,e^r)\) with \(d>e\).  Write
\(B_1=[1,q]\), \(B_2=[q+1,q+r]\).  There is a single nontrivial potential
\(\xch_1\), and
\begin{equation}\label{eq:two-length-A}
 A_1=\xch_1-qre,\qquad A_0=A_2=0.
\end{equation}
Consequently,
\begin{align}
 \kappa_j&=S_{1,j}-\overline S_1+\frac{\xch_1-qre}{q},
 &&1\le j\le q,
 \label{eq:two-length-first}\\
 \kappa_j&=S_{2,j}-\overline S_2-\frac{\xch_1-qre}{r},
 &&q<j\le q+r.
 \label{eq:two-length-second}
\end{align}
The opposite block averages exhibit directly how the zero-sum constraint couples the
two local systems.
\end{corollary}

\begin{proof}
Formula~\eqref{eq:two-length-A} is Theorem~\ref{thm:two-potential} with
\(b_1=q\) and \(\tau_1=re\).  Apply Theorem~\ref{thm:local-global} to the two blocks.
\end{proof}

\begin{example}[Shape \(\mu=(3,3,1)\)]\label{ex:local-global}
At the calibrated reverse row superstandard point \(Q_*\), formula
\eqref{eq:kappa-star} gives
\[
 \kappa^*(\mu)=(-4,2,2).
\]
There are two blocks, \(B_1=\{1,2\}\) and \(B_2=\{3\}\).  Since
\(\xch_1(Q_*)=0\), the first block prefix is
\[
 A_1=-b_1\sum_{j>b_1}\mu_j=-2.
\]
Within \(B_1\), the local difference is
\(C_1=\kappa_2^*-\kappa_1^*=6\).  Hence
\[
 (S_{1,1},S_{1,2})=(0,6),\qquad \overline S_1=3.
\]
The local and global contributions are therefore
\[
 \kappa^{\mathrm{loc}}|_{B_1}=(-3,3),
 \qquad
 \kappa^{\mathrm{glob}}_1=\frac{A_1}{2}=-1,
\]
whose sum is \((-4,2)\).  On the singleton block,
\(A_2-A_1=2\), so the local part is zero and the global part is \(2\).  Thus
\[
 (-4,2,2)=\underbrace{(-3,3,0)}_{\text{blockwise local}}
            +\underbrace{(-1,-1,2)}_{\text{block-global}}.
\]
This small example shows that neither the local differences nor the block prefixes alone
recover \(\kappa\) when repeated and distinct row lengths occur together.
\end{example}

\subsection{Inverse weights and concluding remarks}

The correction vector now completes the inverse transformation begun in
Theorem~\ref{thm:inverse-tabloids}.

\begin{corollary}[The normalized weight under inversion]\label{cor:inverse-beta}
Let \(\AMBC(w)=(P,Q,\rho)\), so that
\(P=\bar P\) and \(Q=\E(\bar Q)\).  Then
\begin{equation}\label{eq:inverse-beta-final}
 \beta(w^{-1})
 =\Dom_{Q,P}\!\bigl(-\beta(w)+\kappa(P,Q)\bigr)
  +\kappa(Q,P).
\end{equation}
Equivalently,
\begin{equation}\label{eq:inverse-beta-rho}
 \beta(w^{-1})=\Dom_{Q,P}(-\rho(w))+\kappa(Q,P).
\end{equation}
Once the first stable index \(N_0^-\) of \(w^{-1}\) is known, its original partition is
recovered by
\begin{equation}\label{eq:inverse-lambda}
 \lambda^- =\mu(N_0^--2)-\beta(w^{-1}).
\end{equation}
\end{corollary}

\begin{proof}
By \eqref{eq:AMBC-inverse},
\(\rho(w^{-1})=\Dom_{Q,P}(-\rho(w))\).  Apply
\(\beta=\rho+\kappa\) to the reversed tabloid pair \((Q,P)\), and then substitute
\(\rho(w)=\beta(w)-\kappa(P,Q)\).  Formula~\eqref{eq:inverse-lambda} is the consistency
identity \eqref{eq:lambda-from-beta}.
\end{proof}

The comparison can therefore be summarized as follows.  The AMBC weight \(\rho\)
contains the movable stream altitude.  The correction \(\kappa(P,Q)\) depends only on
the two tabloids.  Its local differences are controlled by dual gaps and local charge,
while its block prefixes are controlled by the difference of the dual and AMBC path
potentials.  Window shifts and cyclic Knuth moves are precisely the elementary edge
increments of those potentials.  This places the original stable-window data, the
transformation rules, and the common monodromy in a single coordinate framework.

\section{Cell-theoretic consequences}\label{sec:cells}

We now translate the preceding comparison into the language of affine
Kazhdan--Lusztig cells.  All cell statements below are made in a fixed index stratum;
equivalently, one may translate to the index-zero affine Weyl group.  We use the
convention in which fixing the AMBC insertion tabloid gives a right cell and fixing the
recording tabloid gives a left cell \cite{CPY18,CLP18}.  A connected component generated
by affine Knuth moves inside such a right cell is a Kazhdan--Lusztig molecule.  Rotation
edges will be mentioned separately, since they change the index by one.

\subsection{Two affine coordinate systems on a cell}

\begin{proposition}[Change of coordinates on a cell]\label{prop:cell-coordinates}
Let
\[
 \AMBC(w)=(P,Q,\rho),
 \qquad
 \RSd(w)=(\bar P,\bar Q,\beta,N_0).
\]
Then
\begin{equation}\label{eq:cell-coordinate-change}
 \bar P=P,\qquad \bar Q=\E(Q),\qquad
 \beta=\rho+\kappa(P,Q).
\end{equation}
Consequently:
\begin{enumerate}[label=(\roman*),leftmargin=2.2em]
\item within a fixed index stratum, fixing \(\bar P\) describes the same right cell as
fixing \(P\), and fixing \(\bar Q\) describes the same left cell as fixing \(Q\);
\item for fixed \((P,Q)\), the AMBC and dual weight fibers differ by the integral affine
translation \(\rho\mapsto\rho+\kappa(P,Q)\);
\item this translation preserves the index hyperplane and intertwines every admissible
altitude-packet deformation.
\end{enumerate}
Thus, after the common index is fixed, \((Q,\rho)\) and \((\E(Q),\beta)\) are two
coordinate systems on the same right cell, with \(\kappa\) as their
tabloid-dependent change of origin.
\end{proposition}

\begin{proof}
The tabloid identities are \eqref{eq:tabloid-comparison}, and the weight identity is the
definition of \(\kappa\) in Corollary~\ref{cor:kappa-independent}.  The cell statements
follow from the AMBC description of left and right cells.  Finally,
\(\sum_i\kappa_i=0\) by \eqref{eq:kappa-sum}, and Proposition~\ref{prop:altitude}
shows that the two weights receive the same packet increment.
\end{proof}

\subsection{Gauge equivalence of the two cocycles}

Fix an insertion tabloid \(P\).  If \(e:Q\to Q'\) is an oriented affine Knuth or
rotation edge, identify the dual endpoint with \(\E(Q')\).  Denote the corresponding
AMBC and dual increments by \(c_\rho(e)\) and \(c_\beta(e)\).

\begin{theorem}[Cocycle gauge equivalence]\label{thm:gauge-equivalence}
On every recording-fiber edge, including an affine Knuth edge or a rotation edge,
\begin{equation}\label{eq:gauge-equivalence}
 c_\beta(e)-c_\rho(e)
   =\kappa(P,Q')-\kappa(P,Q).
\end{equation}
Hence the two vector-valued cocycles differ by the coboundary of \(\kappa\).  They
define isomorphic integral affine local systems and have equal periods on every closed
path.  On the affine Knuth subgraph of a fixed-index right cell, their common holonomy
is \(G_\mu\).
\end{theorem}

\begin{proof}
This is Lemma~\ref{lem:fiber-gauge}.  On a closed path the coboundary telescopes to
zero, so the periods coincide.  If the path uses only affine Knuth edges, it remains in
one index stratum, and Theorem~\ref{thm:monodromy} identifies the common image with
\(G_\mu\).
\end{proof}

The restriction to affine Knuth edges in the last sentence is essential.  The enlarged
graph with rotations still has equal \(\beta\)- and \(\rho\)-periods, but it has extra
periods that change the index.  For example, \(n\) successive right rotations form a
closed tabloid path and have common period \(\mu\), by
Corollary~\ref{cor:iterated-shifts}; this vector does not belong to \(G_\mu\) because
its coordinate sum is \(n\).

In this language, the two-potential statistic \(\xch_s\) is a scalar gauge: applying the
block-prefix functional \(\pi_s\) to \eqref{eq:gauge-equivalence} produces an exact
scalar cocycle.  Formula~\eqref{eq:A-xch} is its normalized primitive, and
Theorem~\ref{thm:local-global} reconstructs the full vector gauge from its block
primitives and its equal-row differences.

\subsection{Molecules and the charge obstruction}

Let \(\mu'\) be the conjugate partition and set
\begin{equation}\label{eq:dmu}
 d_\mu:=\gcd(\mu'_1,\mu'_2,\ldots).
\end{equation}
Extend the local-charge convention by setting \(\lch_i(Q)=0\) whenever
\(\mu_i\ne\mu_{i+1}\), and define the tabloid charge by
\begin{equation}\label{eq:tabloid-charge}
 \operatorname{charge}(Q):=\sum_{i=1}^{\ell-1}i\,\lch_i(Q).
\end{equation}
This is the normalization of \cite[Definition~8.3]{CLP18}.  Their
tabloid-connectivity theorem \cite[Theorem~8.6]{CLP18} says that
\begin{equation}\label{eq:charge-components}
 Q\text{ and }Q'\text{ lie in the same Knuth component}
 \quad\Longleftrightarrow\quad
 \operatorname{charge}(Q)\equiv\operatorname{charge}(Q')\pmod{d_\mu}.
\end{equation}

\begin{proposition}[Dual description of a molecule]\label{prop:dual-molecule}
Fix \(P\) and two valid dual data points
\((P,\E(Q),\beta)\) and \((P,\E(Q'),\beta')\) of shape \(\mu\) and of the same index.
They can lie in the same affine Knuth molecule only if
\begin{equation}\label{eq:molecule-charge}
 \operatorname{charge}(Q)\equiv\operatorname{charge}(Q')\pmod{d_\mu}.
\end{equation}
When \eqref{eq:molecule-charge} holds, choose any affine Knuth path
\(\gamma:Q\rightsquigarrow Q'\), and let \(\beta_\gamma\) be the result of transporting
\(\beta\) by the dual cocycle.  Then the two points lie in the same affine Knuth
molecule if and only if
\begin{equation}\label{eq:molecule-weight}
                         \beta'-\beta_\gamma\in G_\mu.
\end{equation}
The condition is independent of the chosen path \(\gamma\).
\end{proposition}

\begin{proof}
Condition \eqref{eq:molecule-charge} is exactly the existence criterion
\eqref{eq:charge-components} for an affine Knuth path between the two recording
tabloids.  Once one such path is fixed, every other lift with the same tabloid endpoints
differs from it by a closed affine Knuth path.  Theorem~\ref{thm:monodromy} says that the
possible changes of the lifted dual weight are precisely \(G_\mu\).  Therefore a lift
ending at \(\beta'\) exists exactly when \eqref{eq:molecule-weight} holds.  Replacing
\(\gamma\) changes \(\beta_\gamma\) by an element of the same lattice, proving path
independence.
\end{proof}

\subsection{Finiteness and the one-molecule criterion}

\begin{theorem}[Consequences for affine cells]\label{thm:cell-consequences}
Let \(w\) have shape \(\mu\), and work in its fixed index stratum.
\begin{enumerate}[label=(\roman*),leftmargin=2.2em]
\item If \(\mu\) is rectangular, then \(G_\mu=0\), and every affine Knuth molecule of
shape \(\mu\) is finite.
\item If \(\mu\) is nonrectangular, then \(G_\mu\ne0\), and every such molecule is
infinite.
\item The right cell containing \(w\) is a single affine Knuth molecule if and only if
the row lengths of \(\mu\) are pairwise distinct.
\item The tabloid Knuth graph has \(d_\mu\) connected components, cyclically indexed by
charge modulo \(d_\mu\); under the dual correspondence these are the possible recording
components for \(\bar Q=\E(Q)\).
\end{enumerate}
All four statements may be expressed with \(\beta\) in place of \(\rho\), with no change
to the lattice or the component structure.
\end{theorem}

\begin{proof}
For a rectangle, equality of adjacent coordinates throughout the single equal-row block,
together with coordinate sum zero, forces \(G_\mu=0\).  There are finitely many
recording tabloids and then at most one reachable weight over each, proving (i).  For a
nonrectangle, the lattice in \eqref{eq:monodromy-lattice} has positive rank; the
monodromy construction realizes all its elements, so iteration gives infinitely many
weights in the molecule.  These are also Corollary~7.6 and Theorem~7.28 of
\cite{CLP18}.

Statement (iii) is \cite[Corollary~8.7]{CLP18}: repeated row lengths leave a weight
obstruction not removable by monodromy, whereas pairwise distinct row lengths eliminate
both that obstruction and the tabloid-component obstruction.  Statement (iv) is
\cite[Theorem~8.6]{CLP18}.  Finally, Theorem~\ref{thm:gauge-equivalence} transfers every
path and every period from \(\rho\)-coordinates to \(\beta\)-coordinates.
\end{proof}

The special-shape formulas of Section~\ref{sec:comparison} now have a cell-theoretic
reading.  Rectangles have only the blockwise local correction and trivial holonomy;
pairwise distinct rows have only the global potential correction and a single molecule
per right cell; two-row-length shapes are the first family in which a nontrivial
monodromy direction couples two internally corrected blocks.  Thus the local--global
decomposition of \(\beta-\rho\) separates precisely the two sources of affine cell
complexity: repeated-row dominance inside blocks and monodromy between blocks.

\end{document}